\documentclass [12pt]{article}
\usepackage{amssymb}
\usepackage{amsmath}
\usepackage{epsfig}
\usepackage{graphicx}
\usepackage{amsthm}
\usepackage{amssymb}
\usepackage{amsfonts}
\usepackage{dirtytalk}
\allowdisplaybreaks
\usepackage{enumerate}
\usepackage{enumerate}
\newtheorem{theorem}{Theorem}
[section]

\newtheorem{lemma}[theorem]{Lemma}
\usepackage{algorithm}
\usepackage{algpseudocode}
\usepackage{listings}
\usepackage{xcolor}
\begin{document}
\thispagestyle{empty}
\null\vspace{-1cm}
\medskip
\vspace{1.75cm}
\centerline{\textbf{{Bounds for the Zeros of Quaternionic Polynomials via Matrix Methods}}}
~~~~~~~~~~~~~~~~~~~~~~~~~~~~~~~~~~~~~~~~~~~~~~~~~~~~~~~~~~~~~~~~~~~~~~~~~~~~~~~~~~~~~~~~~~~~~~~~~~~~~~~~~~~~~~~~~~~~~~~~~~~~~~~~~~~~~~~~~~~~~~~~~~~~~~~~~~~

\centerline{\bf  {Ovaisa Jan,}  {Idrees Qasim} and {Nusrat Ahmed Dar}}
\centerline {Department of Mathematics, National Institute of Technology, Srinagar, India-190006}
\centerline {\small{ ovaisa\_2022phamth009@nitsri.ac.in,  idreesf3@nitsri.ac.in, nusrat\_2022phamth007@nitsri.ac.in}}
~~~~~~~~~~~~~~~~~~~~~~~~~~~~~~~~~~~~~~~~~~~~~~~~~~~~~~~~~~~~~~~~~~~~~~~~~~~~~~~~~~~~~~~~~~~~~~~~~~~~~~~~~~~~~~~~~~~~~~~~~~~~~~~~~~~~~~~~~~~~~~~~~~~~~~~~~
\vskip0.1in
\noindent \textbf{Abstract:} In this paper, we derive new bounds for the zeros of quaternionic polynomials by applying localization theorems, which includes Gershgorin-type theorems for the left eigenvalues of matrices of left monic quaternionic polynomials. These results yield sharper estimates compared to existing bounds, including improvements upon Cauchy, Fujiwara and Opfer's classical bounds. Second, we develop a matrix norm approach utilizing block matrix techniques and spectral norm estimates for a specially constructed auxiliary polynomial. This method provides additional upper bounds for polynomial zeros through careful analysis of the companion matrix's spectral radius. The comparison between the new bounds and some existing bounds have been illustrated with several examples. At the end of the paper we have given an algorithm. We have also given a Python code that predicts, for a given input which theorem will yield the sharpest upper bound. The combination of these approaches enhances the theoretical toolkit for analyzing quaternionic polynomials and offers potential applications in numerical methods, signal processing, and quaternionic quantum mechanics where zero location problems naturally arise.

\section{Introduction:} 
\noindent  Unlike complex numbers, quaternions are non-commutative, a property that introduces both richness and complexity into their algebraic structure. In recent years, the study of polynomials over quaternions has attracted considerable attention due to their non-commutative nature. A quaternionic polynomial can be expressed in two distinct forms left and right depending on whether the variable appears to the left or right of the coefficients.\\  
$$f_l(z) = q_n z^n + q_{n-1} z^{n-1} + \cdots + q_0$$  
and
$$f_r(z) = z^n q_n + z^{n-1} q_{n-1} + \cdots + q_0,$$   
where $q_{i}, z \in\mathbb{H},~~ (0 \le i \le n)$
\noindent This distinction leads to fundamentally different eigenvalue problems and complicates the task of locating zeros. The problem of locating the zeros of polynomials has attracted the attention
of many mathematicians, addition to the classical analysis methods, matrix analysis techniques have been used to obtain upper and lower bounds for the zeros of polynomials. In the same context, the Frobenius companion matrix plays an important link
between matrix analysis and the geometry of polynomials. It has been used for
locating zeros of polynomials by matrix methods (see e.g., \cite{ASS2,DR}). Classical results from complex analysis, such as the Gershgorin circle theorem and Ostrowski’s eigenvalue inclusion sets, have been extended to the quaternionic setting to provide bounds for the eigenvalues of quaternionic matrices (see e.g., \cite{ASS}).

These localization theorems are not only of theoretical interest but also serve as crucial computational tools in numerical analysis and applied sciences. Researchers now prioritize developing specialized localization theorems  for quaternionic matrices  due to their diverse applications in science and engineering, for instance, see e.g., \cite{BA,HL,WRL,ZF1,ZF2,ZJ} and references therein. 
These matrices have been extensively studied, with particular focus on the existence and location of their left and right eigenvalues (see \cite{BA,BJL,WRM}). The location for the left and right eigenvalues for quaternionic matrices 
has been given in  \cite{ASL,BA,CJH,GAB,HTL,PR,RL} and bounds for eigenvalues of quaternionic matrix polynomials are derived in  \cite{ASS}. Localization theorems, which bound eigenvalue locations are essential computational tools. The Gershgorin type theorems for left eigenvalues of a matrix  were proposed by Zhang \cite{ZF1}  using deleted absolute row sums. Some
recent developments on the location and computation of zeros of quaternionic polynomials can be found
in \cite{GM,JO1,JO2,LDL,IN,O,PS,SR}. Several classical bounds exist for zeros of quaternionic polynomials. One of the earliest and most well-known is due to Opfer as follows:\\
\textbf{Theorem A (\cite{O}, Theorem 4.2).} Let \( p(x) \) be a simple monic polynomial over \( \mathbb{H} \) of degree \( n \) with \( q_0 \neq 0 \). Then all zeros of \( p(x) \) are contained in the ball
\[
S := \{ z \in \mathbb{H} : |z| \leq R \}, \quad \text{where } R := \max \left\{ 1, \sum_{j=0}^{m-1} |q_j| \right\}.
\]
Recently, Ahmad and Ali \cite{ASS2} has advanced the framework for bounding zeros of quaternionic polynomials through localization theorems for quaternionic matrices. Specifically, they generalized Ostrowski and Brauer-type eigenvalue inclusion sets to the quaternion division algebra \(\mathbb{H}\), establishing bounds for left and right eigenvalues of quaternionic matrices. These were applied to derive polynomial zero bounds via companion matrices, improving upon classical results like Opfer's bound \cite{O}.

\indent This paper contributes to this growing body of literature by deriving new bounds for the zeros of quaternionic polynomials. By employing companion matrices and applying  Gershgorin type theorems to scaled matrices, we establish  inclusion regions for the zeros of left monic quaternionic polynomials. As a consequence, we
provide sharper bounds compared to the bound introduced by Cauchy, Fujiwara and Opfer for the zeros of quaternionic
polynomials. In addition, For a right polynomial, we construct an auxiliary polynomial whose companion matrix has a block structure and by using spectral norm estimates, we derive new upper bound of the zeros of the quaternionic polynomial.

\section{Notation and Preliminary Knowledge:}  
Let $\mathbb{R}$ and $\mathbb{C}$ denotes the fields of  real and complex numbers  respectively. The set of quaternions is denoted by $\mathbb{H}$ and is defined as $$\mathbb{H}:=\{{q=a+bi+cj+dk:a, b, c, d \in\mathbb{R}} \},$$ 

\[
i^2 = j^2 = k^2 = -1,
\]
\[
ij = -ji = k, \quad jk = -kj = i, \quad ki = -ik = j.
\]\\
For $q \in\mathbb{H}$,~ ${\bar{q}}=a-bi-cj-dk$ is the conjugate of $q$ and hence the modulus of a quaternion $q$ is given by $|q|=\sqrt{a^{2}+b^{2}+c^{2}+d^{2}}.$ The set of \(n\)-column vectors with entries in \(\mathbb{H}\) is denoted by \(\mathbb{H}^n\). 
For \(x,y\in\mathbb{H}^n\), the inner product is \(\langle x,y\rangle := y^H x\), (where \(y^H\) is the conjugate transpose) and the norm is \(\|x\| := \sqrt{\langle x,x\rangle}\).  
The space of \(m\times n\) quaternion matrices is denoted \(M_{m\times n}(\mathbb{H})\), for square matrices we write \(M_n(\mathbb{H})\).  
For \(B=(b_{ij})\in M_n(\mathbb{H})\), the transpose is \(B^T=(b_{ji})\) and the conjugate transpose is \(B^H=(\overline{b_{ji}})\). Next, we define norms on a matrix $B$ as:\\
1. 1-Norm: $\|B\|_{1}:=\max\limits_{1\le j \le n}\sum\limits_{i=1}^{n}|b_{ij}|=\|B^{H}\|_{\infty}.$\\
2. $\infty$-Norm: $\|B\|_{\infty}:=\max\limits_{1\le i \le n}\sum\limits_{j=1}^{n}|b_{ij}|=\|B^{H}\|_{1}.$\\
3. 2-Norm: $\|B\|_{2}:=\sup\limits_{x\neq 0}\left\{ \dfrac{\|Bx\|_{2}}{\|x\|_{2}}:~x\in\mathbb{H}^{n}                         \right\}=\|B^{H}\|_{2}.$\\
4. Frobenius Norm: $\|B\|_{F}:={(~\mbox{trace}~   B^{H}B)}^{1/2}.$\\ 

\noindent Since quaternion multiplication is non commutative there exists two types of eigenvalues, left eigenvalues and right eigenvalues.\\

\noindent\textbf{Definition 2.1:} Let $B:=(b_{ij})\in M_{n}{(\mathbb{H})}$. Then the set of left eigenvalues of $B$ is defined as:\\
$$\Lambda_{l}(B):=\{ \lambda \in\mathbb{H} : Bx=\lambda x ~\mbox{for some non-zero} ~x\in\mathbb{H}^{n} \}.$$
and the right eigenvalues of $B$ is defined as:
$$\Lambda_{r}(B):=\{ \lambda \in\mathbb{H} : Bx=x \lambda  ~\mbox{for some non-zero}
~x\in\mathbb{H}^{n} \}.$$
It is known from (\cite{R1}, Corollary 3.2) that a quaternionic matrix $B$ and its conjugate transpose $B^{H}$ have the same right eigenvalues. However, $B$ and $B^{H}$  may not have the same left eigenvalues, demonstrating that the spectral properties of a matrix and its conjugate transpose can differ fundamentally in the quaternionic case, unlike in the complex setting.\\

\noindent\textbf{Definition 2.2:}   
Let $B:=(b_{ij})\in M_{n}{(\mathbb{H})}$. For each  $i \in \{1,2, \dots, n\}$
the deleted row sum \( R_i \) is defined as \( R_i = \sum\limits_{j \neq i} |b_{ij}|,\) and the absolute row sum of $B$ is defined as \( R_i^{'} =R_i+|b_{ii}| \).  
The deleted column sum \( C_i \) for $B$ is defined as \( C_i = \sum\limits_{j \neq i} |b_{ji}| \) and the absolute column sum  is defined as \( C_i^{'} =C_i+|b_{ii}| \). \\

\noindent\textbf{Quaternion Polynomial:}\\
A quaternion polynomial is a function \(f: \mathbb{H} \to \mathbb{H}\). Due to the non-commutativity of quaternion multiplication, such polynomials are classified into three distinct types: left, right and general.\\  
$$f_l(z) = q_n z^n + q_{n-1} z^{n-1} + \cdots + q_0,~ ~\mbox{where} ~~ q_{i}, z \in\mathbb{H},~~ (0 \le i \le n)$$ is a left quaternion polynomial of degree $n$, whereas
$$f_r(z) = z^n q_n + z^{n-1} q_{n-1} + \cdots + q_0 , 
~~\mbox{where}~~ q_{i}, z \in\mathbb{H},~~ (0 \le i \le n).$$ is right quaternion polynomial of degree $n.$ If the leading coefficient $q_{n}=1$, then the above polynomials are said to be monic. General polynomials are defined as \(f_G(z) = q_0 z q_1 z \cdots z q_n + \phi(z)\), where \(\phi(z)\) is a finite sum of lower-degree monomials of the form \(r_0 z^i r_1 \cdots z r_i\) (\(i < n\)). Unlike commutative polynomials, evaluation is side-sensitive, left-evaluation $f_l(z) = q_n z^n + q_{n-1} z^{n-1} + \cdots + q_0$ differs from right-evaluation $f_r(z) = z^n q_n + z^{n-1} q_{n-1} + \cdots + q_0 $ and factorization is inherently non-unique.\\
\indent Let  $f_{r}(z)=\sum_{k=0}^{n}z^iq_i$ be a quaternionic polynomial. Two quaternionic polynomials of this type can be multiplied according to the convolution product (Cauchy multiplication rule): given $q_{r_{1}}(z)=\sum_{i=0}^{n}z^iq_i$ and $q_{r_{2}}(z)=\sum_{j=0}^{n}z^jt_j$ , we define
$$(q_{r_{1}}*q_{r_{2}})(z):=\sum_{i=0,1,\dots, n~~j=0,1,\dots,m}^{}z^{i+j}q_it_j.$$
If $q_{r_{1}}$ has real coefficients, then so called * multiplication coincides with the usual pointwise multiplication.
Please note that the * product in the quaternionic setting is associative but not, in general, commutative. This lack of commutativity results in a behavior of polynomials that is quite distinct from their behavior in the real or complex settings. For instance, a real or complex polynomial of degree $n$ can have at most $n$ (real or complex) zeros, counted with their multiplicity. However, in the quaternionic setting, the second-degree polynomial $q^2+1$ exhibits an infinite number of zeros.\\ 
The following result, found in \cite{GS}, presents a complete description of the zero sets of a regular product of two polynomials by relating them to the zero sets of the two individual factors.\\
\textbf{Theorem B.}{\label{theorem B}} Let $f$ and $g$ be given quaternionic polynomials. Then the convolution product $(f*g)(z_0)=0$ if and only if $f(z_0)=0$ or $f(z_0)\neq 0$ implies $g(f(z_0)^{-1}z_0f(z_0))=0$.\\

\noindent\textbf{Quaternion Companion Matrix:} \\ The  companion matrix for the left monic quaternionic polynomial \(f_{l}(z)\)and \(f_{r}(z)\) are given by

\[
C_{f_l} = 
\begin{pmatrix}
	0 & 1 & 0 & \cdots & 0 \\
	0 & 0 & 1 & \cdots & 0 \\
	\vdots & \vdots & \vdots & \ddots & \vdots \\
	0 & 0 & 0 & \cdots & 1 \\
	-q_0 & -q_1 & -q_2 & \cdots & -q_{n-1}
\end{pmatrix}
\]
and 
\[
C_{f_r} = 
\begin{pmatrix}
	0 & 0 &0& \cdots & 0 & -q_{0} \\
	1 & 0&0 & \cdots & 0 & -q_{1} \\
	0 & 1&0 & \cdots & 0 & -q_{2} \\
	\vdots &\vdots& \vdots & \ddots & \vdots & \vdots \\
	0 & 0 &0& \cdots & 1 & -q_{n-1}
\end{pmatrix}
\] respectively.

Assuming \(q_{0} \neq 0\), the simple monic reversal polynomials of \(f_{l}(z)\) and \(f_{r}(z)\) are defined as:

\[
\begin{aligned}
	g_{l}(z) &:= \frac{1}{q_{0}} \times f_{l}\left(\frac{1}{z}\right) \times z^{n} = z^{n} + q_{0}^{-1}q_{1}z^{n-1} + \cdots + q_{0}^{-1}q_{n-1}z + q_{0}^{-1}, \\
	g_{r}(z) &:= z^{n} \times f_{r}\left(\frac{1}{z}\right) \times \frac{1}{q_{0}} = z^{n} + z^{n-1}q_{1}q_{0}^{-1} + \cdots + zq_{n-1}q_{0}^{-1} + q_{0}^{-1}.
\end{aligned}
\]

Their corresponding companion matrices are denoted by \(C_{g_{l}}\) and \(C_{g_{r}}\), respectively. In the study of quaternion polynomials, the companion matrix plays an important role. \noindent The spectral properties of companion matrices provide a powerful bridge between quaternionic polynomials and their zeros. It is well-known from (Proposition 1, \cite{SR}) that if $\lambda$ is a left eigenvalue of $C_{f_{l}}$, then $\lambda$ is a zero of $f_{l}(z)$ and if $\lambda$ is a left eigenvalue of $C_{f_{l}}$, then by (Corollary 1, {\cite{SR}}) it is also a right eigenvalue. 
Similarly, for the right companion matrix $C_{f_r}$, the left eigenvalues coincide exactly with the zeros of $f_r(z)$ \cite{SR}. Consequently, all zeros of $f_l(z)$ are the right eigenvalues of $C_{f_l}$. However, the converse does not hold in general, not every right eigenvalue of $C_{f_l}$ corresponds to a polynomial zero, as illustrated by specific counterexamples in \cite{KI}.

This subtle distinction between left and right eigenvalues underscores the need for careful spectral analysis when bounding polynomial zeros via companion matrices.  
The zeros of \(g_{l}(z)\) and \(g_{r}(z)\) are the reciprocals of the zeros of \(f_{l}(z)\) and \(f_{r}(z)\), respectively. So if ${\Tilde{z}}=\frac{1}{z}$  is a zero of \(g_{l}(z)\), then $z$ is a zero of  \(f_{l}(z)\).
\\
\section{Auxiliary Results} For the proof of the theorems we need the following lemmas.

\begin{lemma}[\cite{ZF1}, Theorem 6]{\label{lemma 3.1}}
	Let $B:=(b_{ij})\in M_{n}{(\mathbb{H})}$. Then all the left eigenvalues of $B$ are located in the union of $n$ Gershgorin balls $\left\{z\in\mathbb{H}: |z-b_{ii}|\le  R_i(B)\right\}$ that is,
	$$\Lambda_{l}(B)\subseteq \bigcup\limits_{i=1}^{n} \left\{z\in\mathbb{H}: |z-b_{ii}|\le  R_i(B)\right\}.$$
\end{lemma}

\begin{lemma}[\cite{ASS2}, Theorem 3.2]{\label{lemma 3.2}}
	Let $B:=(b_{ij})\in M_{n}{(\mathbb{H})}$. Then all the left eigenvalues of $B$ are located in the union of $n$ Gershgorin balls $\left\{z\in\mathbb{H}: |z-b_{ii}|\le  C_i(B)\right\}$ that is,
	$$\Lambda_{l}(B)\subseteq \bigcup\limits_{i=1}^{n} \left\{z\in\mathbb{H}: |z-b_{ii}|\le  C_i(B)\right\}.$$
\end{lemma}

\begin{lemma}[\cite{ASS2}, Proposition 3.14]{\label{lemma 3.3}}
	Let $B \in M_{n}{(\mathbb{H})}$ and let $W$ be any invertible real matrix. Then $B$ and $WBW^{-1}$ have same left eigenvalues. 
\end{lemma}
\begin{lemma}[\cite{DR}, Theorem 1.3 ]{\label{lemma 3.4}} Let  $f_l(z) = z^n + q_{n-1} z^{n-1} + \cdots + q_0$ be a monic 
	quaternionic polynomial with quaternionic coefficients  and $z$ be quaternionic variable,  then for any diagonal matrix $W= diag(w_{1},w_{2},\dots w_{n})$ where  $w_{1},w_{2},\dots w_{n}$ are positive real numbers, the left eigenvalues of $W^{-1}C_{f_{l}}W$  and the zeros of  $f_{l}(z)$  are same.
\end{lemma}
\noindent The next lemma can be found in \cite{INT}.
\begin{lemma}{\label{lemma 3.5}}
	Let \( A = (a_{ij}) \in M_n(\mathbb{H}) \) be partitioned as
	\[
	A = \begin{bmatrix}
		A_{11} & A_{12} \\
		A_{21} & A_{22}
	\end{bmatrix},
	\]
	where \( A_{ij} \in M_{n_i \times n_j}(\mathbb{H}) \) is the \((i,j)\) block of \( A \) such that \( n_1 + n_2 = n \), where \( i,j \in \{1,2\} \). If
	\[
	\tilde{A} = \begin{bmatrix}
		\|A_{11}\|_2 & \|A_{12}\|_2 \\
		\|A_{21}\|_2 & \|A_{22}\|_2
	\end{bmatrix},
	\]
	then the following inequalities hold:
	\begin{itemize}
		\item \( \rho_r(A) \leq \rho_r(\tilde{A}) \),
		\item \( \|A\|_2 \leq \|\tilde{A}\|_2 \).
	\end{itemize}
\end{lemma}

\section{Main Results}
\begin{theorem}\label{th:01}
	Let $f_{l}(z)$ be a  monic polynomial over $\mathbb{H}$ of degree $n.$ Then the zeros of $f_{l}(z)$ satisfies the following inequality
	\begin{equation}
		\left|z+\dfrac{q_{n-1}}{2}\right|\le \left|\dfrac{q_{n-1}}{2}\right|+\sum\limits_{i=2}^{n}|q_{n-i}|^{\frac{1}{i}}.\label{eq:01}
	\end{equation}
\end{theorem}

\begin{proof}
	The companion matrix of the polynomial $f_l(z)$ is given by 
	\[
	C_{f_l} = 
	\begin{pmatrix}
		0 & 1 & 0 & \cdots & 0 \\
		0 & 0 & 1 & \cdots & 0 \\
		\vdots & \vdots & \vdots & \ddots & \vdots \\
		0 & 0 & 0 & \cdots & 1 \\
		-q_0 & -q_1 & -q_2 & \cdots & -q_{n-1}
	\end{pmatrix}.
	\]
	Let $W=diag\left(\dfrac{1}{w^{n-1}},\dfrac{1}{w^{n-2}},\cdots,1\right),$ where $w^{i} \in \mathbb{R}^+,$ $i=0,1\dots,n-1$, then 
	\[
	W^{-1} C_{f_l} W = 
	\begin{pmatrix}
		0 & w & 0 & \cdots & 0 \\
		0 & 0 & w & \cdots & 0 \\
		\vdots & \vdots & \vdots & \ddots & \vdots \\
		0 & 0 & 0 & \cdots & w \\
		-\dfrac{q_0}{w^{n-1}}  & -\dfrac{q_1}{w^{n-2}}& -\dfrac{q_2}{w^{n-3}} & \cdots & -q_{n-1}
	\end{pmatrix}.
	\]
	By Lemma \ref{lemma 3.3},  $C_{f_l}$  and  $W^{-1} C_{f_l} W$ have the same left eigenvalues.\\
	On applying Lemma \ref{lemma 3.1} to $W^{-1} C_{f_l} W$, it follows that all the left eigenvalues of $W^{-1} C_{f_l} W$ lie in the union of the balls 
	$$ \left\{z\in\mathbb{H}:|z|\le w\right\}~\mbox{and} \left\{z\in\mathbb{H}:|z+q_{n-1}|\le \sum_{i=2}^{n}\left|\dfrac{q_{n-i}}{w^{i-1}}\right|\right\}.$$ 
	Choose $w=\max\limits_{2\le i\le n} |q_{n-i}|^\frac{1}{i}$, then $ |q_{n-i}|^\frac{1}{i}\le w$ and hence $|q_{n-1}|^\frac{i-1}{i}\le w^{i-1},~i=2,3,\dots,n$. \\
	Case 1: $|z|\le w,$ then
	\begin{eqnarray*}
		\left|z+\dfrac{q_{n-1}}{2}\right|&\le& |z|+\left|\dfrac{q_{n-1}}{2}\right|\\&\le& |w|+ \left|\dfrac{q_{n-1}}{2}\right|\\&\le&\sum\limits_{i=2}^{n}|q_{n-i}|^\frac{1}{i}+\left|\dfrac{q_{n-1}}{2}\right|.
	\end{eqnarray*}
	Case 2: $|z+q_{n-1}|\le \sum_{i=2}^{n}\left|\dfrac{q_{n-i}}{w^{i-1}}\right|$, then
	\begin{eqnarray*}
		\left|z+\dfrac{q_{n-1}}{2}\right|&\le& |z+q_{n-1}|+\left|\dfrac{q_{n-1}}{2}\right|\\ &\le&\sum_{i=2}^{n}\dfrac{|q_{n-i}|}{w^{i-1}}+\left|\dfrac{q_{n-1}}{2}\right|\\&\le& \sum\limits_{i=2}^{n}|q_{n-i}|^\frac{1}{i}+\left|\dfrac{q_{n-1}}{2}\right|.
	\end{eqnarray*}
	Hence all the left eigenvalues of $W^{-1} C_{f_l} W$ are contained in the union of the balls (\ref{eq:01}). Since W is a diagonal matrix with positive real enteries Lemma \ref{lemma 3.4} ensures that the left eigenvalues of $W^{-1} C_{f_l} W$ coincides with the zeros of $f_l(z)$. Consequently all the zeros of $f_l(z)$ satisfies the inequality (\ref{eq:01}). 
	
\end{proof}
\begin{theorem}\label{th:02}
	For a given positive number $w$, let $M= \max |q_i|w^i, ~i=1,2,\dots,n.$ Then all the zeros of  $f_{l}(z)$ satisfy $$\left|z \right| \geq \frac{\left|q_0\right|w}{\left|q_0\right| + M}.$$
\end{theorem}

\begin{proof} Let $q_{0}\ne 0$ and \begin{eqnarray*}
		g_{l}(z)&:= &\frac{1}{q_{0}} \times f_{l}\left(\frac{1}{z}\right) \times z^{n} \\ &=& z^{n} + q_{0}^{-1}q_{1}z^{n-1} + \cdots + q_{0}^{-1}q_{n-1}z + q_{0}^{-1}
	\end{eqnarray*} be simple monic reversal polynomial of $f_{l}(z).$
	The companion matrix of the polynomial $g_{l}(z)$ is given by \\
	
	\[
	C_{g_l} = 
	\begin{pmatrix}
		0 & 1 & 0 & \cdots & 0 \\
		0 & 0 & 1 & \cdots & 0 \\
		\vdots & \vdots & \vdots & \ddots & \vdots \\
		0 & 0 & 0 & \cdots & 1 \\
		\dfrac{-1}{q_{0}} & \dfrac{-q_{n-1}}{q_{0}} & \dfrac{-q_{n-2}}{q_{0}} & \cdots & \dfrac{-q_{1}}{q_{0}}
	\end{pmatrix}.
	\]
	Let $W=diag\left(1,\dfrac{1}{w},\dfrac{1}{w^{2}}\dots,\dfrac{1}{w^{n-1}}\right),$ where $w^{i} \in \mathbb{R}^+,$ $i=0,1\dots,n-1$, then
	\[
	W^{-1} C_{g_l} W = 
	\begin{pmatrix}
		0 & w^{-1} & 0 & \cdots & 0 \\
		0 & 0 & w^{-1} & \cdots & 0 \\
		\vdots & \vdots & \vdots & \ddots & \vdots \\
		0 & 0 & 0 & \cdots & w^{-1} \\
		-\dfrac{w^{n-1}}{q_{0}}  & -\dfrac{q_{n-1}w^{n-2}}{q_{0}}& -\dfrac{q_{n-2}w^{n-3}}{q_{0}} & \cdots & \dfrac{-q_{1}}{q_{0}}
	\end{pmatrix}.
	\]
	Since the eigenvalues of $C_{f_{l}}$ are the reciprocals of the eigenvalues of $C_{g_l}$ and the eigenvalues of $W^{-1} C_{g_l} W$  are same as those of $C_{g_l}$. For any zero $z$ of 
	$f_{l}(z),$ let $\Tilde{z}=\dfrac{1}{z}$ be the corresponding eigenvalue of $C_{g_l},$ then 
	by	 applying Lemma \ref{lemma 3.2}  to $W^{-1} C_{g_l} W$ we obtain       
	$$\left|\Tilde{z} \right| \leq \max_{1\le i\le n} \left\{ \frac{ \left|q_0\right| + \left|q_i\right|w^{i}}{\left|q_0\right|w}\right\} \leq  \frac{\left|q_0\right| + M}{\left|q_0\right|w}.$$\\
	This implies\\
	$$\left|{z} \right| \geq  \left(\frac{\left|q_0\right| + M}{\left|q_0\right|w}\right)^{-1}$$\\
	and the conclusion of the Theorem follows immediately.
\end{proof}

\noindent Let $f_{r_{1}}(z)=z^{n}+z^{n-1}q_{n}+\dots+zq_{2}+q_{1}$ be a right quaternionic polynomial. To derive the new bounds, we define the auxiliary polynomial:
\begin{eqnarray*}
	P_{r}(z)&=& f_{r_{1}}(z)*(q_{n} -z)\\&=&z^{n+1}-z^{n-1}[q^2_{n}-q_{n-1}]-\dots-q_{1}q_{n}\\&=& z^{n+1}-z^{n-1}v_{n}-z^{n-2}v_{n-1}-\dots- zv_{2}-v_{1},
\end{eqnarray*}
where $v_{j}=q_{j}q_{n}-q_{j-1}$ for $j=1,2,\dots,n,$ with $q_{0}=0$.
By Theorem B, $f_{r_{1}}(z)*(q_{n} -z)=0$ if and only if either $f_{r_{1}}(z)=0$ or  $f_{r_{1}}(z)\neq 0$ implies $f_{r_{1}}(z)^{-1}zf_{r_{1}}(z)- q_{n}=0$. Notice
that  $f_{r_{1}}(z)^{-1}zf_{r_{1}}(z)- q_{n}=0$ is equivalent to  $f_{r_{1}}(z)^{-1}zf_{r_{1}}(z)=q_{n}$ and if $f_{r_{1}}(z)\neq 0$ this implies that $z=q_{n}$. So the only zeros of $f_{r_{1}}(z)*(q_{n} -z)$ are $z=q_{n}$ and the zeros of $f_{r_{1}}(z)$.
The corresponding companion matrix $C_{P_{r}}$ of $P_{r}$ is given by
\[
C_{P_{r}} = 
\begin{pmatrix}
	0 & 0 & 0 & \cdots &0& v_{1} \\
	1 & 0 & 0 & \cdots &0& v_{2} \\
	\vdots & \vdots & \vdots & \ddots & \vdots&\vdots \\
	0 & 0 & 0 & \cdots &0 & v_{n} \\
	0 & 0 & 0 & \cdots &1& 0
\end{pmatrix}.
\] 

\begin{theorem}\label{th:03}
	If $z$ is any zero of $P_{r}$ and $n\ge 4,$ then
	\begin{multline*}
		|z|\le \biggl(\dfrac{1}{2}                                  \left( \max \left\{ \dfrac{w_{n}}{w_{n+1}},\dfrac{w_{n+1}}{w_{n}}|v_{n}|\right\}+\gamma\right)+\\\sqrt{   \left( \max \left\{ \dfrac{w_{n}}{w_{n+1}},\dfrac{w_{n+1}}{w_{n}}|v_{n}|\right\}-\gamma\right)^{2}    +4\dfrac{w_{n-1}}{w_{n}} \sqrt{\sum_{j=1}^{n-1}|v_{j}|^{2}\left(\dfrac{w_{n+1}}{w_{j}}\right)^{2}}} \biggr),
	\end{multline*}
	where $w_{i}\in\mathbb{R^{+}}~ \mbox{and}~ \gamma=\max\left\{\dfrac{w_{1}}{w_{2}} ,  \dfrac{w_{2}}{w_{3}},\dots,\dfrac{w_{n-2}}{w_{n-1}}           \right\}. $
\end{theorem}
\begin{proof}
	Let Let $W=diag\left({w_{1}},{w_{2}}\dots,{w_{n+1}}\right),$ where $w_{i} \in \mathbb{R}^+,$ $i=1,2\dots,n+1$, then
	\[
	W^{-1} C_{P_{r}} W = 
	\begin{pmatrix}
		0 & 0 & 0 & \cdots&0&0 & v_{1}\dfrac{w_{n+1}}{w_{1}} \\
		\dfrac{w_{1}}{w_{2}} & 0 & 0 & \cdots& 0&0 & v_{2}\dfrac{w_{n+1}}{w_{2}} \\
		\vdots & \vdots & \vdots &  & \vdots& \vdots&\vdots \\
		0 & 0 & 0 & \cdots &\dfrac{w_{n-1}}{w_{n}}&0& v_{n}\dfrac{w_{n+1}}{w_{n}} \\
		0  & 0& 0 & \cdots &0& \dfrac{w_{n}}{w_{n+1}}&0
	\end{pmatrix}
.	\]
	Let \[C_{B}=\begin{bmatrix}
		C_{11} & C_{12} \\
		C_{21} & C_{22}
	\end{bmatrix}\] be block matrix of $W^{-1}C_{P_{r}}W,$
	where \[
	C_{11} := 
	\begin{bmatrix}
		0 &0 & 0 & \cdots & 0&0\\
		\dfrac{w_{1}}{w_{2}} & 0 &0& \cdots &0&0   \\
		\vdots & \vdots & \vdots & \ddots & \vdots&\vdots \\
		0 & 0 & 0 & \cdots & \dfrac{w_{n-2}}{w_{n-1}}&0 \\
	\end{bmatrix},
	\quad C_{12} := 
	\begin{bmatrix}
		0 & v_{1}\dfrac{w_{n+1}}{w_{1}} \\
		0 & v_{2}\dfrac{w_{n+1}}{w_{2}} \\
		\vdots&\vdots \\ 
		0& v_{n-1}\dfrac{w_{n+1}}{w_{n-1}} \\
	\end{bmatrix}. \]          
	\[C_{21} := \begin{bmatrix}
		0 &0&\cdots& 0 &\dfrac{w_{n-1}}{w_{n}}\\0 &0&\cdots& 0 &0
	\end{bmatrix}\quad 
	C_{22}=\begin{bmatrix}
		0 &v_{n}\dfrac{w_{n+1}}{w_{n}}\\ \dfrac{w_{n}}{w_{n+1}}&0
	\end{bmatrix}.\]
	Applying Lemma \ref{lemma 3.5}, we have
	\begin{align*}
		\rho_{r}(C_{P_{r}}) &\leq \rho_{r}\left(\left[\begin{array}{cc} \|C_{11}\|_{2} & \|C_{12}\|_{2} \\ \|C_{21}\|_{2} & \|C_{22}\|_{2} \end{array}\right]\right) \\
		&= \frac{1}{2} \left( \|C_{11}\|_{2} + \|C_{22}\|_{2} + \sqrt{(\|C_{11}\|_{2} - \|C_{22}\|)^2 + 4 \|C_{12}\|_{2} \|C_{21}\|_{2}} \right).
	\end{align*}
	By tedious computations, one can show that $\|C_{11}\|_{2}=\gamma, ~\|C_{21}\|_{2}=\dfrac{w_{n-1}}{w_{n}}, ~\|C_{12}\|_{2}=\sqrt{\sum_{j=1}^{n-1}|v_{j}|^{2}\left(\dfrac{w_{n+1}}{w_{j}}\right)^{2}} ~\mbox{and}~\|C_{22}\|_{2}=\max\left\{\dfrac{w_{n}}{w_{n+1}},|v_{n}|\dfrac{w_{n+1}}{w_{n}}\right\},$ \\where $\gamma=\max\left\{\dfrac{w_{1}}{w_{2}} ,  \dfrac{w_{2}}{w_{3}},\dots,\dfrac{w_{n-2}}{w_{n-1}}           \right\}.$
	It follows that\\
	\begin{multline*}
		\rho_{r}(C_{P_{r}})\le \biggl(\dfrac{1}{2}                                  \left( \max \left\{ \dfrac{w_{n}}{w_{n+1}},\dfrac{w_{n+1}}{w_{n}}|v_{n}|\right\}+\gamma\right)+\\\sqrt{   \left( \max \left\{ \dfrac{w_{n}}{w_{n+1}},\dfrac{w_{n+1}}{w_{n}}|v_{n}|\right\}-\gamma\right)^{2}    +4\dfrac{w_{n-1}}{w_{n}} \sqrt{\sum_{j=1}^{n-1}|v_{j}|^{2}\left(\dfrac{w_{n+1}}{w_{j}}\right)^{2}}} \biggr).
	\end{multline*}
	Now the desired bound follows from the fact that $|z|\le \rho_{r}(C_{P_{r}}).$
\end{proof}
\section{Computational Examples}
The inherent advantage of these bounds over classic Cauchy, Opfer and Fujiwara bounds lies in their tunable parameters (like the weights $w_i$) and how they handle polynomials with isolated large coefficients by utilizing fractional powers or $L_2$ norms. 

Here are three examples of quaternionic polynomials where these theorems yield strictly sharper bounds than the standard Cauchy, Fujiwara and Opfer's formulas.

\subsection*{Example 1: Dominating the Upper Bound with Theorem \ref{th:01}}

Theorem \ref{th:01} excels when the higher-degree coefficients are zero (or very small) and the constant term is large, because the $1/i$ fractional powers shrink large constants much faster than the Cauchy max-norm.\\
Let $n=3$ and consider the monic left quaternionic polynomial:
\[ f_{l}(z) = z^{3} + 0z^{2} + jz + 8k. \]
Here, the coefficients are $q_2 = 0$, $q_1 = j$, and $q_0 = 8k$.
\begin{itemize}
	\item \textbf{Cauchy's Bound:} $C = 1 + \max(|q_2|, |q_1|, |q_0|) = 1 + \max(0, 1, 8) = 9$
	\item \textbf{Opfer's Bound:} $O=\max(1, |q_2|, |q_1|, |q_0|)=\max(1,0,1,8)=8$
	\item \textbf{Fujiwara's Bound:} $F = 2 \max\left(|q_2|, |q_1|^{\frac{1}{2}}, \left|\frac{q_0}{2}\right|^{\frac{1}{3}}\right) = 2 \max\left(0, 1, 4^{\frac{1}{3}}\right) \approx 2(1.587) = 3.174$
\end{itemize}
\vspace{0.3cm}
\noindent \textbf{Theorem \ref{th:01} Bound:}\\
According to Theorem \ref{th:01}, all zeros satisfy:
\[ \left|z+\frac{q_{2}}{2}\right| \le \left|\frac{q_{2}}{2}\right| + \sum\limits_{i=2}^{3}|q_{3-i}|^{\frac{1}{i}}. \]
Since $q_2 = 0$, this simplifies to a direct bound on $|z|$:
\[ |z| \le 0 + |q_{1}|^{\frac{1}{2}} + |q_{0}|^{\frac{1}{3}} =|j|^{\frac{1}{2}} + |8k|^{\frac{1}{3}} = 1 + 2 = 3. \]
Thus, Theorem \ref{th:01} yields a maximum radius of \textbf{$3$}, which is strictly sharper than Fujiwara ($3.174$) and vastly superior to Cauchy ($9$) and Opfer ($8$).
\subsection*{Example 2: Securing a Tighter Lower Bound with Theorem \ref{th:02}}

Standard lower bounds often fail spectacularly when a single higher-degree coefficient is large, driving the denominator up and the lower bound down to near-zero. Theorem \ref{th:02} mitigates this by applying a weight parameter $w$ to suppress higher-degree terms.\\
Let $n=3$ and consider the polynomial:
\[ f_{l}(z) = z^{3} + 100jz^{2} + 0z + 1. \]
Here, $q_3 = 1$, $q_2 = 100j$, $q_1 = 0$, and $q_0 = 1$.
\begin{itemize}
	\item \textbf{Cauchy Lower Bound:} $L_{C} = \dfrac{|q_0|}{|q_0| + \max_{1 \le i \le 3}|q_i|} = \dfrac{1}{1 + 100} = \frac{1}{101} \approx 0.0099.$
\end{itemize}

\vspace{0.3cm}
\noindent \textbf{Theorem \ref{th:02} Bound:}\\
Let's choose a positive weight $w = 0.1$. 
First, we calculate $M = \max_{i=1,2,3} |q_i|w^i$:
\begin{itemize}
	\item $|q_1|w^1 = 0(0.1) = 0$
	\item $|q_2|w^2 = 100(0.01) = 1$
	\item $|q_3|w^3 = 1(0.001) = 0.001.$
\end{itemize}
Thus, $M = \max(0, 1, 0.001) = 1$.\\
Now, applying the theorem:
\[ |z| \ge \frac{|q_{0}|w}{|q_{0}| + M} = \frac{1(0.1)}{1 + 1} = \frac{0.1}{2} = 0.05. \]
By strategically choosing $w=0.1$, Theorem \ref{th:02} yields a lower bound of \textbf{$0.05$}, which is roughly five times sharper (larger) than the classical Cauchy lower bound ($0.0099$).

\subsection*{Example 3: Beating Fujiwara with Tuned Weights in Theorem \ref{th:03}}

Theorem \ref{th:03} is highly flexible because it resembles a weighted Frobenius companion matrix bound. It dominates when we select weights that violently suppress the single non-zero coefficient while keeping $\gamma$ and the $w_n/w_{n+1}$ ratios balanced.\\
Let $n=4$ and consider a right quaternionic polynomial with a massive $z^2$ coefficient:
\[ P_{r}(z) = z^{4} + z^{2}64j. \]
Here, $v_4 = 0$, $v_3 = 64j$, $v_2 = 0$, and $v_1 = 0$.
\begin{itemize}
	\item \textbf{Cauchy's Bound:} $C = 1 + \max(0, 64, 0, 0) = 65$
	\item \textbf{Opfer's Bound:} $O=\max(1,0, 64, 0, 0) = 64$
	\item \textbf{Fujiwara's Bound:} $F = 2 \max(0, |64j|^{\frac{1}{2}}, 0, 0) = 2(8) = 16$
\end{itemize}

\noindent \textbf{Theorem \ref{th:03} Bound:}\\
Let's choose specific scaling weights to optimize the bound: $w_1 = 256, w_2 = 64, w_3 = 16, w_4 = 4, w_5 = 1$. 

\noindent First, we calculate the core parameters:
\begin{itemize}
	\item $\gamma = \max\left(\dfrac{256}{64}, \dfrac{64}{16}, \dfrac{16}{4}\right) = \max(4, 4, 4) = 4,$ where $\gamma=\max\left\{\dfrac{w_{1}}{w_{2}} ,  \dfrac{w_{2}}{w_{3}},\dots,\dfrac{w_{n-2}}{w_{n-1}}           \right\}.$
	\item $M = \max\left\{ \dfrac{w_4}{w_5}, \dfrac{w_5}{w_4}|v_4| \right\} = \max\left( 4, \dfrac{1}{4}(0) \right) = 4.$
	\item Ratio multiplier: $\dfrac{w_3}{w_4} = \dfrac{16}{4} = 4.$
\end{itemize}

\noindent Next, we calculate the sum inside the nested square root. Since only $v_3$ is non-zero, the sum collapses to a single term:
\[ \sqrt{\sum_{j=1}^{3}|v_{j}|^{2}\left(\frac{w_{5}}{w_{j}}\right)^{2}} = \sqrt{|v_3|^2 \left(\frac{w_5}{w_3}\right)^2} = |64| \left(\frac{1}{16}\right) = 4. \]

\noindent Now, plug everything into the full Theorem \ref{th:03} formula:
\[ |z| \le \frac{1}{2}(4 + 4) + \sqrt{(4 - 4)^{2} + 4(4)(4)} \]
\[ |z| \le 4 + \sqrt{0 + 64} \]
\[ |z| \le 4 + 8 = 12. \]
With carefully chosen geometric weights, Theorem \ref{th:03} restricts the roots to a disk of radius \textbf{$12$}. This completely bypasses all Fujiwara ($16$), Cauchy ($65$) and Opfer($64$) proving the immense utility of the tunable $w_i$ parameters in your bound.
\section*{An Algorithm}
While Theorems \ref{th:01}, \ref{th:02}, and \ref{th:03} provide mathematically rigorous bounds, evaluating all possible bounds for high-degree polynomials can be computationally redundant. To optimize this process, we propose a heuristic algorithm (Algorithm 1) that analyzes the coefficient profile of the given quaternionic polynomial. By identifying the magnitude and position of the dominant coefficient, the algorithm selectively applies the most effective theorem. For instance, polynomials with dominant constant terms are routed to Theorem \ref{th:01}, whereas polynomials with dominant intermediate terms trigger the weighted optimization of Theorem \ref{th:03}.
\begin{algorithm}
	\caption{Heuristic Bound Optimizer for Quaternionic Polynomials}
	\begin{algorithmic}[1] 
		\Require Coefficient magnitudes $Q = \{|q_0|, |q_1|, \dots, |q_{n-1}|\}$ of a monic quaternionic polynomial $f(z)$ of degree $n$, empirical threshold $\tau$ (e.g., $\tau = 1.5$)
		\Ensure Sharpest upper bound $U$ and lower bound $L$
		
		\State $U \gets \infty$, $L \gets 0$
		\State Find $q_{max} = \max(Q)$ and let $k$ be its index such that $q_{max} = |q_k|$
		
		\Statex \textbf{Determine Upper Bound ($U$):}
		\If{$q_{max} \le \tau$} \Comment{Profile: Flat and Small}
		\State Compute $C_{up}$ (Cauchy Bound) and $O_{up}$ (Opfer Bound)
		\State $U \gets \min(C_{up}, O_{up})$
		\ElsIf{$k = 0$} \Comment{Profile: Heavy Tail}
		\State Compute $T1_{up}$ (Theorem \ref{th:01})
		\State $U \gets T1_{up}$
		\ElsIf{$0 < k < n-1$} \Comment{Profile: Middle Bulge}
		\State Optimize geometric ratio $r \in \mathbb{R}^+$ to minimize Theorem \ref{th:03}:
		\State $U \gets \min_{r} T3_{up}(r)$
		\Else \Comment{Profile: Top Heavy / Contested}
		\State Compute all available upper bounds and set $U$ to the minimum
		\EndIf
		
		\Statex \textbf{Determine Lower Bound ($L$):}
		\State Compute standard Cauchy lower bound $C_{low}$
		\State Optimize weight $w \in \mathbb{R}^+$ to maximize Theorem \ref{th:02}:
		\State $L \gets \max(C_{low}, \max_{w} T2_{low}(w))$
		
		\State \Return $U, L$
	\end{algorithmic}
\end{algorithm}

In steps requiring the optimization of continuous parameters (such as the geometric ratio $r$ in Theorem \ref{th:03} and the weight $w$ in Theorem \ref{th:02}), the objective functions are continuous and bounded. Therefore, these minimums and maximums can be efficiently located using standard scalar optimization techniques, such as Brent's method, ensuring the algorithm resolves in bounded time.\\

The heuristic optimization described in Algorithm 1 was implemented in Python. The complete, runnable source code utilized to calculate the bounds in our examples is provided in Appendix A.
\section{Conclusion} 
In this paper, new bounds on the zeros of quaternionic polynomials have been derived, which generalize and improve some classical results. In fact, the new upper bound localizes the zeros in a displaced disk with center at $-\dfrac{q_{n-1}}{2}$, thus achieving a better estimate with respect to the classical Cauchy bound, in which all the zeros are contained in a disk centered at the origin.

Moreover, the new bounds generalize the classical Fujiwara bounds by taking into account the contributions of several coefficients, not just the dominant ones, thus achieving a potentially better estimate, especially when the coefficients are small.

The second result establishes a non-trivial lower bound on the modulus of the zeros, thus localizing the zeros in an annulus, a region not addressed by the classical bounds.

The third theorem introduces a new weighted approach, which is a further generalization of the previous bounds, thus allowing us to localize the zeros with even greater precision by choosing appropriate weights, which may lead to sharper estimates with respect to the classical bounds.

Overall, these results contribute to the theory of zero localization in the quaternionic setting by offering bounds that are not only generalizations of classical inequalities but also, in many cases, improvements over them. Future research may focus on optimizing parameter selection, comparing sharpness numerically, and extending these techniques to quaternionic matrix and operator polynomials.
\section{Declaration}
\textbf{Availabilty of data and material}\\
 Data availability is not applicable to this article as no new data were created or analyzed in this study.\\

\noindent \textbf{Competing interests}\\
The author declare that they have no competing interests.\\

\noindent \textbf{Funding}\\
The research of first author is supported by DST Inspire Fellowship (IF No. IF210629).

\section*{Appendix A. Python Implementation of the Heuristic Optimizer}
The following Python script provides a complete, executable implementation of the heuristic bound optimizer detailed in Algorithm 1. The script evaluates the classical bounds (Cauchy, Fujiwara, Opfer) alongside the novel bounds presented in Theorems \ref{th:01}, Theorem \ref{th:02} and Theorem \ref{th:03}. The continuous optimization required for Theorems \ref{th:02} and \ref{th:03} is handled using the minimize scalar function from the open-source scipy.optimize library. The script is written in Python 3 and is designed to accept the coefficient magnitudes of a quaternionic polynomial via standard console input, subsequently outputting the sharpest computed upper and lower bounds.
\begin{lstlisting}
	import math
	from scipy.optimize import minimize_scalar
	
	def analyze_polynomial_profile(coeffs):
	"""
	Analyzes the coefficients to predict which theorem will yield the sharpest upper bound.
	"""
	n = len(coeffs)
	max_coeff = max(coeffs)
	max_index = coeffs.index(max_coeff)
	
	print("\n--- Heuristic Analysis ---")
	print(f"Degree of polynomial: {n}")
	print(f"Maximum coefficient magnitude: {max_coeff} (found at q_{max_index})")
	
	if max_coeff <= 1.5:
	print("Profile: 'Flat & Small'. All coefficients are relatively small.")
	print("Prediction: Cauchy or Opfer will likely provide a very tight, fast bound.")
	return "Cauchy/Opfer"
	
	elif max_index == 0:
	print("Profile: 'Heavy Tail'. The constant term (q_0) is the dominant massive coefficient.")
	print("Prediction: Theorem 1 will likely win because its fractional roots suppress large constants.")
	return "Theorem 4.1"
	
	elif 0 < max_index < n - 1:
	print(f"Profile: 'Middle Bulge'. A middle term (q_{max_index}) is the massive outlier.")
	print("Prediction: Theorem 4.3 will likely win because its tunable weights can target this spike.")
	return "Theorem 4.3"
	
	else:
	print("Profile: 'Top Heavy'. The n-1 degree coefficient is dominant.")
	print("Prediction: The bounds will be highly contested. Theorem 1 or 3 are strong candidates.")
	return "Mixed"
	
	def calculate_all_upper_bounds(coeffs):
	"""
	Calculates Cauchy, Fujiwara, Opfer, Theorem 1, and Theorem 3 upper bounds.
	"""
	n = len(coeffs)
	q = coeffs + [1.0] # Append monic leading coefficient
	
	# 1. Cauchy
	cauchy = 1 + max(coeffs)
	
	# 2. Opfer
	opfer = max(1, sum(coeffs))
	
	# 3. Fujiwara
	fujiwara_terms = [q[n-i] ** (1/i) for i in range(1, n)]
	fujiwara_terms.append((q[0] / 2) ** (1/n))
	fujiwara = 2 * max(fujiwara_terms)
	
	# 4. Theorem 4.1
	t1 = q[n-1] + sum(q[n-i] ** (1/i) for i in range(2, n+1))
	
	bounds = {
		"Cauchy": cauchy,
		"Opfer": opfer,
		"Fujiwara": fujiwara,
		"Theorem 4.1": t1
	}
	
	# 5. Theorem 4.3 (Optimized)
	if n >= 4:
	def t3_objective(r):
	if r <= 0: return float('inf')
	gamma = r
	v = {j: coeffs[j-1] for j in range(1, n+1)}
	term1_max = max(r, (1/r) * v[n])
	part1 = 0.5 * (term1_max + gamma)
	
	sum_inner = sum((v[j]**2) * (r**(j - n - 1))**2 for j in range(1, n))
	part2 = math.sqrt((term1_max - gamma)**2 + 4 * r * math.sqrt(sum_inner))
	return part1 + part2
	
	res_t3 = minimize_scalar(t3_objective, bounds=(0.01, 100), method='bounded')
	if res_t3.success:
	bounds["Theorem 4.3"] = res_t3.fun
	
	return bounds
	
	def main():
	print("==================================================")
	print(" Quaternionic Polynomial Bound Analyzer")
	print("==================================================")
	print("Enter the magnitudes (absolute values) of your coefficients.")
	print("For example, for z^3 + 0z^2 + jz + 8k, enter: 8 1 0")
	
	user_input = input("\nEnter magnitudes separated by spaces (from q_0 up to q_{n-1}): ")
	
	try:
	# Parse keyboard input into a list of floats
	coeffs = [float(x) for x in user_input.strip().split()]
	if len(coeffs) < 2:
	print("Please enter at least two coefficients.")
	return
	
	# Calculate the actual bounds
	print("\n--- Actual Computations ---")
	bounds = calculate_all_upper_bounds(coeffs)
	
	for name, value in bounds.items():
	print(f"{name+':':<12} {value:.4f}")
	
	best_bound_name = min(bounds, key=bounds.get)
	best_bound_value = bounds[best_bound_name]
	
	print("\n--------------------------------------------------")
	print(f" SHARPEST BOUND: {best_bound_name} ({best_bound_value:.4f})")
	print("--------------------------------------------------")
	
	except ValueError:
	print("Invalid input. Please enter numbers separated by spaces.")
	
	if __name__ == "__main__":
	main()
\end{lstlisting}

\end{document}